\documentclass{jcmlatex}
\begin{document}

\def\JCMvol{xx}
\def\JCMno{x}
\def\JCMyear{200x}
\def\JCMreceived{Month xx, 200x}
\def\JCMrevised{}
\def\JCMaccepted{}
\def\JCMonline{}
\setcounter{page}{1}



\def\cgeq{\raisebox{-1mm}{$\;\stackrel{>}{\sim}\;$}}
\def\cequiv{\raisebox{-1mm}{$\;\stackrel{=}{\sim}\;$}}
\def\cleq{\raisebox{-1mm}{$\;\stackrel{<}{\sim}\;$}}
\def\T{\mathop{\cal T}}
\def\esssup{\mathop{\rm esssup}}
\def\qed{\hfill$\fbox{}$}
\catcode`@=11
\newskip\plaincentering \plaincentering=0pt plus 1000pt minus 1000pt
\def\@plainlign{\tabskip=0pt\everycr={}} 
\def\eqalignno#1{\displ@y \tabskip=\plaincentering
  \halign to\displaywidth{\hfil$\@plainlign\displaystyle{##}$\tabskip=0pt
    &$\@plainlign\displaystyle{{}}##$\hfil\tabskip=\plaincentering
    &\llap{$\hbox{\rm\@plainlign##}$}\tabskip=0pt\crcr
    #1\crcr}}



\markboth{C. LI AND Y. LI}
 {Minimal Ger\v{s}gorin tensor eigenvalue inclusion set and its numerical approximation}

\title{Minimal Ger\v{s}gorin tensor eigenvalue inclusion set and its numerical approximation}

\author{Chaoqian Li~and ~Yaotang Li
               \thanks{School of Mathematics and Statistics, Yunnan
University, Kunming, Yunnan,  P. R. China 650091\\
Email: lichaoqian@ynu.edu.cn,~liyaotang@ynu.edu.cn}}

\maketitle

\begin{abstract}
For a complex tensor $\mathcal {A}$, Minimal Ger\v{s}gorin tensor
eigenvalue inclusion set of $\mathcal {A}$ is presented, and its
sufficient and necessary condition is given. Furthermore, we study
its boundary by the spectrums of the equimodular set and the
extended equimodular set for $\mathcal {A}$. Lastly, for an
irreducible tensor, a numerical approximation to Minimal
Ger\v{s}gorin tensor eigenvalue inclusion set is given.
\end{abstract}



\section{Introduction}

Let $N=\{1,2,\ldots, n\}.$ For a complex (real) order $m$ dimension
$n$ tensor $\mathcal {A}=(a_{i_1\cdots i_m})$ (written $\mathcal
{A}\in \mathbb{C}^{[m,n]}$ ($\mathbb{R}^{[m,n]}$), respectively),
where
\[a_{i_1\cdots i_m}\in \mathbb{C}~ (\mathbb{R}), ~i_j=1,\ldots,n,~j=1,\ldots, m.\]
we call a complex number $\lambda$ an eigenvalue of $\mathcal {A}$
and a nonzero complex vector $x$ an eigenvector of $\mathcal {A}$
associated with $\lambda$, if
\begin{equation}\label{eigvalue}\mathcal {A}x^{m-1}=\lambda
x^{[m-1]},\end{equation} where $\mathcal {A}x^{m-1}$ and $
x^{[m-1]}$ are vectors, whose $i$th components are
\[(\mathcal {A}x^{m-1})_i=\sum\limits_{i_2,\ldots,i_m\in N}
a_{ii_2\cdots i_m}x_{i_2}\cdots x_{i_m}\] and
\[(x^{[m-1]})_i=x_i^{m-1},\] respectively. Note that there are other
definitions of eigenvalue and eigenvectors, such as, H-eigenvalue,
D-eigenvalue and Z-eigenvalue; see \cite{Qi3,Qi5,Qi2}.  Obviously,
the definition of eigenvalue for matrices follows from the case
$m=2$.

Tensor eigenvalues and eigenvectors have received much attention
recently in the literatures \cite{Ch,La1,Ko1,Li,Ng,Qi3,Qi4,Qi1,Wa}.
Many important results on the eigenvalue problem of matrices have
been successfully extended  to higher order tensors; see
\cite{Ch,Ca1,Li1,Ng,Qi3,Qi4,Qi5,Ya,Ya1}. In \cite{Qi3}, Qi
generalized Ger\v{s}gorin eigenvalue inclusion theorem from matrices
to real supersymmetric tensors, which can be easily extended to
generic tensors; see \cite{Ya}.

\begin{theorem} \emph{\cite{Qi3}}
\label{th 1.1} Let $\mathcal {A}=(a_{i_1\cdots i_m})\in
\mathbb{C}^{[m,n]}$ and $\sigma(\mathcal {A})$ be the spectrum of
$\mathcal {A}$, that is, \[\sigma(\mathcal
{A})=\{\lambda\in\mathbb{C}: \mathcal {A}x^{m-1}=\lambda x^{[m-1]},~
x\in \mathbb{C}^n\backslash\{0\}\}.\] Then
\begin{equation}\label{eq 2.1}
\sigma(\mathcal {A})\subseteq \Gamma(\mathcal
{A})=\bigcup\limits_{i\in N} \Gamma_i(\mathcal {A}),
\end{equation}
where \[\Gamma_i(\mathcal {A})=\left\{z\in \mathbb{C}:|z-a_{i\cdots
i}|\leq r_i(\mathcal {A})=\sum\limits_{i_2,\ldots,i_m\in N,\atop
\delta_{ii_2\ldots i_m}=0} |a_{ii_2\cdots i_m}|\right\}\] and
\[\delta_{i_1i_2\cdots i_m}= \left\{ \begin{array}{cc}
\delta_{i_1i_2\cdots i_m}=1,       & if ~i_1=\cdots =i_m,~\\
 0,       & otherwise.
\end{array} \right. \]
\end{theorem}

It is easy to see that when $m=2$, Theorem \ref{th 1.1} reduces to
the well known Ger\v{s}gorin eigenvalue inclusion theorem of
matrices \cite{Ger,Varga}. Here, we call $\Gamma_i(\mathcal {A})$
the $i$-th Ger\v{s}gorin tensor eigenvalue inclusion set. Note that
$\Gamma_i(\mathcal {A})$ is a closed set in the complex plane
$\mathbb{C}$, Hence, $\Gamma(\mathcal {A})$, which consists of the
$n$ sets $\Gamma_i(\mathcal {A})$, is also closed and bounded in
$\mathbb{C}$.

Also in \cite{Qi3}, Qi obtain  another interesting result on
Ger\v{s}gorin tensor eigenvalue inclusion set $\Gamma(\mathcal
{A})$.

\begin{theorem} \emph{\cite{Qi3}}
If $\Gamma_i(\mathcal {A})$ is disjoint with the other
$\Gamma_j(\mathcal {A}),~j\neq i$, then there are exactly
$(m-1)^{n-1}$ eigenvalues which lie in $\Gamma_i(\mathcal {A})$.
Furthermore, if all of $\Gamma_i(\mathcal {A})$,
$i=l_1,l_2,\ldots,l_k$ are connected but disjoint with the other
$\Gamma_j(\mathcal {A})$, $j\neq i$, then there are exactly
$k(m-1)^{n-1}$ eigenvalues which lie in $\bigcup\limits_{i=l_1,l_2,
\ldots, l_k} \Gamma_i(\mathcal {A}) $.
\end{theorem}

In this paper, we also focus on Ger\v{s}gorin tensor eigenvalue
inclusion set. And we present Minimal Ger\v{s}gorin tensor
eigenvalue inclusion set, give its sufficient and necessary
condition, and research its boundary. For an irreducible tensor, we
give a set which approximates to Minimal Ger\v{s}gorin tensor
eigenvalue inclusion set.

\section{Minimal Ger\v{s}gorin tensor eigenvalue inclusion set}
In this section, we present Minimal Ger\v{s}gorin tensor eigenvalue
inclusion set and study its characteristic. First, a lemma is given.

\begin{lemma} \emph{\cite{Ya}} \label{lemma 2.1} Let $\mathcal{A}=(a_{i_1\cdots i_m})\in \mathbb{C}^{[m,n]}$
and $D=diag(d_1,d_2,\ldots,d_n)$ be a diagonal nonsingular matrix.
If
\[\mathcal{B}=(b_{i_1\cdots
i_m})=\mathcal{A}D^{-(m-1)}\overbrace{DD\cdots D}\limits^{m-1},\]
where \[b_{i_1\cdots i_m}=d_{i_1}^{-(m-1)}a_{i_1i_2\cdots
i_m}d_{i_2}\cdots d_{i_m},~ i_1,\ldots, i_m\in N,\] then $
\mathcal{A}, ~\mathcal{B}$ have the same eigenvalues.
\end{lemma}

From Lemma \ref{lemma 2.1}, we obtain the following tensor
eigenvalue inclusion set.

\begin{theorem}
\label{th 2.1} Let $\mathcal {A}=(a_{i_1\cdots i_m})\in
\mathbb{C}^{[m,n]} $ and $x=(x_1,x_2,\ldots,x_n)^T$ be an entrywise
positive vector, i.e., $x=(x_1,x_2,\ldots,x_n)^T>0$. Then
\begin{equation}\label{equation2.1}\sigma(\mathcal {A})\subseteq \Gamma^{r^x}(\mathcal
{A})=\bigcup\limits_{i\in N}\Gamma^{r^x}_i(\mathcal
{A}),\end{equation} where
\[\Gamma^{r^x}_i(\mathcal {A})=\left\{z\in \mathbb{C}: |z-a_{i\cdots
i}|\leq r^x_i(\mathcal {A})= \sum\limits_{i_2,\ldots,i_m\in N,\atop
\delta_{ii_2\cdots i_m}=0} \frac{|a_{ii_2\cdots i_m}|x_{i_2}\cdots
x_{i_m}}{x_{i}^{m-1}} \right\}.\] Furthermore,
\begin{equation}\label{equation2.2}\sigma(\mathcal {A})\subseteq \bigcap\limits_{x>0}\Gamma^{r^x}(\mathcal {A}).\end{equation}
\end{theorem}

\begin{proof} Let $X=diag(x_1,x_2,\ldots,x_n)$ and $\mathcal {B}=\mathcal{A}X^{-(m-1)}\overbrace{XX\cdots
X}\limits^{m-1}=(b_{i_1\cdots i_m})$. It is obvious that $X$ is
nonsingular and $\sigma(\mathcal {B})=\sigma(\mathcal{A})$ from
Lemma \ref{lemma 2.1}. From Theorem \ref{th 1.1}, we have
\[\sigma(\mathcal {B})\subseteq \Gamma (\mathcal
{B})=\bigcup\limits_{i\in N} \Gamma_i(\mathcal {B}),
\]
where \[\Gamma_i(\mathcal {B})=\left\{z\in \mathbb{C}:|z-b_{i\cdots
i}|\leq r_i(\mathcal {B})=\sum\limits_{i_2,\ldots,i_m\in N,\atop
\delta_{ii_2\ldots i_m}=0} |b_{ii_2\cdots i_m}|\right\}\] Note that
$b_{i\cdots i}=a_{i\cdots i}$ and $r_i(\mathcal {B})=r^x_i(\mathcal
{A})$. Hence, $\Gamma_i(\mathcal {B})=\Gamma^{r^x}_i(\mathcal {A})$,
consequently, $\Gamma (\mathcal {B})=\Gamma^{r^x}(\mathcal {A})$.
Therefore,
\[ \sigma(\mathcal{A})=\sigma(\mathcal {B})\subseteq \Gamma
(\mathcal {B})=\Gamma^{r^x}(\mathcal {A}). \]

Furthermore, for any $x>0,~x\in \mathbb{R}^n$, (\ref{equation2.1})
also holds. Hence, (\ref{equation2.2}) follows. \end{proof}

The set $\bigcap\limits_{x>0}\Gamma^{r^x}(\mathcal {A})$ in Theorem
\ref{th 2.1} is of interest theoretically because it provides a set,
containing all the eigenvalues of $\mathcal {A}$, which is called
Minimal Ger\v{s}gorin tensor eigenvalue inclusion set defined as
follows.

\begin{definition} \label{minimal ger} Let $\mathcal {A}=(a_{i_1\cdots i_m})\in
\mathbb{C}^{[m,n]}$. Then \begin{equation} \Gamma^R(\mathcal
{A})=\bigcap\limits_{x \in \mathbb{R}^n, \atop
x>0}\Gamma^{r^x}(\mathcal {A})
\end{equation}
is Minimal Ger\v{s}gorin tensor eigenvalue inclusion set of
$\mathcal {A}$, related to the collection of all weighted sums,
$r^x_i(\mathcal {A})$, where $x=(x_1,x_2,\ldots,x_n)^T>0$.
\end{definition}

Next, a sufficient and necessary condition for the elements
belonging to Minimal Ger\v{s}gorin tensor eigenvalue inclusion set
is provided. Before that, we give some results involving the
$Perron-Frobenius$ theory of nonnegative tensors \cite{Ya}. Given a
real tensor $\mathcal {A}$, we call $\mathcal {A}$ nonnegative,
denoted by $\mathcal {A}\geq 0$, if every of its entries is
nonnegative.

\begin{definition} \cite[Definition 2.1]{Ch} A tensor
$\mathcal {A}=(a_{i_1\cdots,i_m})\in \mathbb{C}^{[m,n]}$ is called
reducible, if there exists a nonempty proper index subset $I\subset
N$ such that \[|a_{i_1\cdots,i_m }|=0,~\forall i_1\in I, ~\forall
i_2,\ldots,i_m\notin I.\] If $\mathcal {A}$ is not reducible, then
we call  $\mathcal {A}$ irreducible.\end{definition}

\begin{lemma} \label{lemma 2.2}\emph{\cite{Ya}}
\label{nonnegative-tensor-pf} If $\mathcal{A} \in
\mathbb{R}^{[m,n]}$ is nonnegative, then the spectral radius
\[\rho(\mathcal{A})=\max\{|\lambda|: \lambda \in \sigma(\mathcal {A})\}\] is an eigenvalue with a nonnegative eigenvector
$x\neq 0$ corresponding to it. Moreover, if $\mathcal{A}$ is
irreducible, then $\rho(\mathcal{A})>0$ and $x$ is positive.
\end{lemma}

\begin{lemma} \label{lemma 2.4} \emph{\cite[Theorem 5.3]{Ya}}
Let $\mathcal {A} \in \mathbb{R}^{[m,n]}$ be nonnegative. Then
\[\rho(\mathcal {A}) =\max\limits_{x\geq 0, x\neq 0}
\min\limits_{x_i>0} \frac{(\mathcal {A}x^{m-1})_i} {x^{m-1}_i}=
\min\limits_{x\geq 0, x\neq 0} \max\limits_{x_i>0} \frac{(\mathcal
{A}x^{m-1})_i} {x^{m-1}_i}.\]
\end{lemma}

\begin{remark}
The first equality is only given in Theorem 5.3 of \cite{Ya}. And
similar to the proof of  Theorem 5.3 of \cite{Ya}, the second is
proved easily.
\end{remark}

From Lemma \ref{lemma 2.4}, we can obtain the following result.

\begin{corollary} \label{corollary 2.1}
Let $\mathcal {A} \in \mathbb{R}^{[m,n]}$ be nonnegative. Then
\[\rho(\mathcal {A}) =  \sup\limits_{x>0} \min\limits_{i\in N}
\frac{(\mathcal {A}x^{m-1})_i} {x^{m-1}_i}=\inf\limits_{x>0}
\max\limits_{i\in N} \frac{(\mathcal {A}x^{m-1})_i} {x^{m-1}_i}.\]
\end{corollary}

From Corollary \ref{corollary 2.1}, we have the following result.

\begin{lemma} \label{lemma 2.5}
Let $\mathcal {A}=(a_{i_1\cdots,i_m})\in \mathbb{C}^{[m,n]}$ and $z$
be any complex number. And let $\mathcal
{B}(z)=(b_{i_1\cdots,i_m})\in \mathbb{R}^{[m,n]}$, where
\[b_{i\cdots i}=-|z-a_{i\cdots
i}|,~b_{ii_2\cdots,i_m}=|a_{ii_2\cdots,i_m}|~for ~ i\in N
~and~\delta_{ii_2\cdots,i_m}=0.\]  Then $\mathcal {B}(z)$ possesses
a real eigenvalue $v(z)$ which has the property that if $\lambda\in
\sigma(\mathcal {B}(z))$, then
\[Re(\lambda) \leq v(z).\] Furthermore, \begin{equation}\label{eq3.1}v(z)=\inf\limits_{x>0}\max\limits_{i\in N} \frac{(\mathcal {B}(z)x^{m-1})_i}{x_i^{m-1}}.
\end{equation}
\end{lemma}

\begin{proof} Let $\mu=\max\limits_{i\in N}|z-a_{i\cdots i}|$, and
$\mathcal {C}=(c_{i_1\cdots,i_m})\in \mathbb{R}^{[m,n]}$, where
\begin{equation} \label{eq 2.2} c_{i\cdots i}=\mu-|z-a_{i\cdots i}|~and~c_{ii_2\cdots,i_m}=|a_{ii_2\cdots,i_m}|, ~i\in N, ~\delta_{ii_2\cdots,i_m}=0.
\end{equation}
Then $B(z)=-\mu\mathcal {I}+\mathcal {C}$, where $\mathcal {C}\geq
0$. From Equality (\ref{eigvalue}), it is obvious that
$\lambda(\mathcal {C})\in \sigma (\mathcal {C})$ if and only if
$-\mu+\lambda(\mathcal {C})\in \sigma (\mathcal {B}(z))$. Moreover,
From Lemma \ref{lemma 2.2}, $\rho(\mathcal {C})$ is an eigenvalue of
$\mathcal {C}$ with a nonnegative eigenvector. Hence,
$-\mu+\rho(\mathcal {C})$ is an eigenvalue of  $\mathcal {B}(z)$
with a nonnegative eigenvector. Let
\begin{equation} \label{eq 2.3} v(z)=-\mu+\rho(\mathcal {C}).\end{equation} Obviously, $v(z)\in \sigma
(\mathcal{B}(z))$ is real. And for any $\lambda \in \sigma(\mathcal
{B}(z)),$ we have
\[Re(\lambda)=Re(\lambda+\mu-\mu)=Re(\lambda+\mu)-\mu\leq |\lambda+\mu|-\mu\leq \rho(\mathcal {C})-\mu=v(z).\]
Furthermore, by Corollary \ref{corollary 2.1}, we have
\[\rho(\mathcal {C})=\inf\limits_{x>0}
\max\limits_{i\in N} \frac{(\mathcal {C}x^{m-1})_i} {x^{m-1}_i}.\]
Then
\begin{eqnarray*}v(z)&=&-\mu+\rho(\mathcal {C})\\
&=&-\mu+\inf\limits_{x>0} \max\limits_{i\in N} \frac{(\mathcal
{C}x^{m-1})_i} {x^{m-1}_i}\\
&=&\inf\limits_{x>0} \max\limits_{i\in N} \frac{((-\mu\mathcal
{I}+\mathcal {C})x^{m-1})_i} {x^{m-1}_i}\\
&=&\inf\limits_{x>0} \max\limits_{i\in N} \frac{(\mathcal
{B}(z)x^{m-1})_i} {x^{m-1}_i}.\end{eqnarray*} The conclusion
follows.
\end{proof}

For $v(z)$ in Lemma \ref{lemma 2.5}, we give the following property.

\begin{proposition}\label{lemma 4.2} Let $\mathcal {A}=(a_{i_1\cdots i_m})\in \mathbb{C}^{[m,n]}$ and $v(z)$ be defined as Lemma \ref{lemma 2.5}. Then
$v(z)$ is uniformly continuous on $\mathbb{C}$.
\end{proposition}

\begin{proof} Let $\mathcal {B}(z)\in
\mathbb{R}^{[m,n]}$ be defined as Lemma \ref{lemma 2.5}. Then,
similar to the proof of Theorem \ref{theorem3.1},

\[v(z)=\inf\limits_{x>0}\max\limits_{i\in N} \frac{(\mathcal
{B}(z)x^{m-1})_i}{x_i^{m-1}}=\inf\limits_{x>0}\max\limits_{i\in N}
\{r^{x}_i(\mathcal {A})-|z-a_{i\cdots i}|\}.\] Note that for any
$z,~\tilde{z}\in \mathbb{C}$ and $x>0,~x\in \mathbb{R}^n$,
\[\begin{array}{lll}
   \max\limits_{i\in N}\{r^{x}_i(\mathcal {A})-|z-a_{i\cdots i}|\}& = & \max\limits_{i\in N}\{r^{x}_i(\mathcal {A})-|z-\tilde{z}+\tilde{z}-a_{i\cdots i}|\} \\
   &\geq&\max\limits_{i\in N}\{r^{x}_i(\mathcal {A})-(|z-\tilde{z}|+|\tilde{z}-a_{i\cdots
   i}|)\}\\
   &=&\max\limits_{i\in N}\{r^{x}_i(\mathcal {A})-|\tilde{z}-a_{i\cdots
   i}|\}-|z-\tilde{z}|,
\end{array}\]
that is,
\[\max\limits_{i\in N}\{r^{x}_i(\mathcal {A})-|\tilde{z}-a_{i\cdots
   i}|\}\leq |z-\tilde{z}|+\max\limits_{i\in N}\{r^{x}_i(\mathcal {A})-|z-a_{i\cdots i}|\}.\]
Hence,
\[\begin{array}{lll}
\inf\limits_{x>0}\max\limits_{i\in N}\{r^{x}_i(\mathcal
{A})-|\tilde{z}-a_{i\cdots i}|\} &\leq& \inf\limits_{x>0}
\left\{|z-\tilde{z}|+\max\limits_{i\in N}\{r^{x}_i(\mathcal
{A})-|z-a_{i\cdots i}|\}\right\}\\
&=&|z-\tilde{z}|+\inf\limits_{x>0}\max\limits_{i\in
N}\{r^{x}_i(\mathcal {A})-|z-a_{i\cdots i}|\},
\end{array}\]
which implies \[v(\tilde{z})-v(z)\leq |z-\tilde{z}|.\] Similarly, we
can obtain
\[v(z)-v(\tilde{z})\leq |z-\tilde{z}|.\]
Therefore, \[|v(z)-v(\tilde{z})|\leq |z-\tilde{z}|.\] This implies
that $v(z)$ is uniformly continuous on $\mathbb{C}$.
\end{proof}

We now establish the sufficient and necessary condition for
$\Gamma^{R}(\mathcal {A})$.

\begin{theorem} \label{theorem3.1}
Let $\mathcal {A}, ~\mathcal {B}(z)$ and $v(z)$ be defined as Lemma
\ref{lemma 2.5}. Then  \begin{equation}\label{eq3.2}z\in
\Gamma^{R}(\mathcal {A})~if ~and ~only ~if ~v(z)\geq
0.\end{equation}
\end{theorem}

\begin{proof} Assume that $z\in \Gamma^{R}(\mathcal {A})$. From
Definition \ref{minimal ger}, we have that for each vector $x>0$,
$z\in \Gamma^{r^x}(\mathcal {A})$. Hence, there is an $i_0\in N$
such that \[|z-a_{i_0\cdots i_0}|\leq r^{x}_{i_0}(\mathcal {A}),\]
equivalently,
\[r^{x}_{i_0}(\mathcal {A})-|z-a_{i_0\cdots i_0}|\geq 0.\]
Note that for any $i\in N$, \[\frac{(\mathcal
{B}(z)x^{m-1})_i}{x_i^{m-1}}=r^{x}_i(\mathcal {A})-|z-a_{i\cdots
i}|.\] Then  $\frac{\left(\mathcal
{B}(z)x^{m-1}\right)_{i_0}}{x_{i_0}^{m-1}}\geq 0$, which implies
that for each vector $x>0$, \[\max\limits_{i\in N} \frac{(\mathcal
{B}(z)x^{m-1})_i}{x_i^{m-1}}\geq 0.\] By Lemma \ref{lemma 2.5},
\[v(z)=\inf\limits_{x>0}\max\limits_{i\in N} \frac{(\mathcal
{B}(z)x^{m-1})_i}{x_i^{m-1}}\geq 0.\]

Conversely, suppose that $v(z)\geq 0$. From Equality (\ref{eq3.1}),
then for each vector $x>0$, there is an $i\in N$ such that \[0\leq
v(z)\leq \frac{(\mathcal
{B}(z)x^{m-1})_i}{x_i^{m-1}}=r^{x}_i(\mathcal {A})-|z-a_{i\cdots
i}|.\]Hence, $|z-a_{i\cdots i}|\leq r^{x}_i(\mathcal {A})$, and then
$z\in \Gamma^{r^x}_i(\mathcal {A})$ which implies $z\in
\Gamma^{r^x}(\mathcal {A})$. Since this inclusion holds for each
vector $x>0$,  we have $z\in \Gamma^R(\mathcal {A})$ from Definition
\ref{minimal ger}.
\end{proof}

\section{Boundary of $\Gamma^{R}(\mathcal {A})$} In Section 2, a
sufficient and necessary condition for the elements belonging to
Minimal Ger\v{s}gorin tensor eigenvalue inclusion set
$\Gamma^{R}(\mathcal {A})$, is given. We in this section focus on
the boundary of $\Gamma^{R}(\mathcal {A})$, and establish
relationships between its boundary (see Definition \ref{boundary}),
the spectrum of the equimodular set for $\mathcal {A}$ (see
Definition \ref{equimodular}),  the spectrum of the extended
equimodular set for $\mathcal {A}$ (see Definition \ref{extended
equimodular}) and $\Gamma^{R}(\mathcal {A})$.

\begin{definition}\label{boundary}
Let $\mathcal{A} \in \mathbb{C}^{[m,n]}$. The boundary of Minimal
Ger\v{s}gorin tensor eigenvalue inclusion set $\Gamma^{R}(\mathcal
{A})$, denoted by $\partial \Gamma^{R}(\mathcal {A})$, is defined by
\[\partial\Gamma^{R}(\mathcal {A})=\overline{\Gamma^{R}(\mathcal {A})}\bigcap \overline{(\Gamma^{R}(\mathcal {A}))'},\]
where $(\Gamma^{R}(\mathcal {A}))'$ is the complement of
$\Gamma^{R}(\mathcal {A})$, and $\overline{\Gamma^{R}(\mathcal
{A})}$ is the closure of $\Gamma^{R}(\mathcal {A})$.
\end{definition}

\begin{definition}\label{equimodular}
Let $\mathcal{A}=(a_{i_1\cdots i_m}) \in \mathbb{C}^{[m,n]}$. The
equimodular set of $\mathcal{A}$, denoted by $\Omega(\mathcal{A})$,
is defined as \[\Omega(\mathcal{A})=\{\mathcal {Q}=(q_{i_1\cdots
i_m}) \in \mathbb{C}^{[m,n]}: q_{i\cdots i}= a_{i\cdots i},
|q_{ii_2\cdots i_m}|= |a_{ii_2\cdots i_m}|, ~i\in N,
\delta_{ii_2\cdots i_m}=0\}.\]
\end{definition}

\begin{definition}\label{extended equimodular}
Let $\mathcal{A}=(a_{i_1\cdots i_m}) \in \mathbb{C}^{[m,n]}$. The
extended equimodular set of $\mathcal{A}$, denoted by
$\hat{\Omega}(\mathcal{A})$, is defined as
\[\hat{\Omega}(\mathcal{A})=\{\mathcal {Q}=(q_{i_1\cdots i_m}) \in
\mathbb{C}^{[m,n]}: q_{i\cdots i}= a_{i\cdots i}, |q_{ii_2\cdots
i_m}|\leq |a_{ii_2\cdots i_m}|, i\in N, \delta_{ii_2\cdots
i_m}=0\}.\]
\end{definition}

A sufficient and necessary condition for the point lying on the
boundary of $\Gamma^{R}(\mathcal {A})$ is given as follows.

\begin{proposition}\label{boundary nec}
Let $\mathcal {A}$ and $v(z)$ be defined as Lemma \ref{lemma 2.5}.
 $z\in \partial \Gamma^{R}(\mathcal {A})$ if and only if $v(z)= 0$,
 and there is a sequence of $\{z_j\}_{j=1}^{\infty} \subseteq (\Gamma^{R}(\mathcal {A}))'$ (i.e., $v(z_j)<0$ for
 all $j\geq 1$) such that $\lim\limits_{j\rightarrow\infty} z_j=z$.
\end{proposition}

\begin{proof} Note that $\Gamma^{R}(\mathcal {A})$ is a compact set in the complex plane.
Hence, from Definition \ref{boundary}, \[\partial\Gamma^{R}(\mathcal
{A})=\overline{\Gamma^{R}(\mathcal {A})}\bigcap
\overline{(\Gamma^{R}(\mathcal {A}))'}=\Gamma^{R}(\mathcal
{A})\bigcap \overline{(\Gamma^{R}(\mathcal {A}))'}.\] Furthermore,
by Theorem \ref{theorem3.1}, we have that
 \begin{equation}\label{eq4.1}z\in \overline{(\Gamma^{R}(\mathcal {A}))'}~if ~and ~only
~if ~v(z)\leq 0.\end{equation} Therefore, if $z\in \partial
\Gamma^{R}(\mathcal {A})$,  then, from (\ref{eq3.2}) and
(\ref{eq4.1}), $v(z)= 0$. Note that $(\Gamma^{R}(\mathcal {A}))'$ is
open and unbounded, then, by (\ref{eq4.1}) and $z\in
\overline{(\Gamma^{R}(\mathcal {A}))'}$, there is a sequence of
$\{z_j\}_{j=1}^{\infty} \subseteq (\Gamma^{R}(\mathcal {A}))' $ such
that $\lim\limits_{j\rightarrow\infty} z_j=z$, where $v(z_j)<0$ for
all $j\geq 1.$

Conversely, it is obvious by assumption that $z\in
\overline{(\Gamma^{R}(\mathcal {A}))'}$ and $z\in
\Gamma^{R}(\mathcal {A})=\overline{\Gamma^{R}(\mathcal {A})}$, that
is, \[z\in \partial \Gamma^{R}(\mathcal {A}).\] The proof is
completed.
\end{proof}

\begin{proposition}\label{pro 4.2}
For any $\mathcal {A}=(a_{i_1\cdots i_m})\in \mathbb{C}^{[m,n]}$ and
any $z\in \mathbb{C}$ with $v(z)=0$, there is a tensor $\mathcal
{Q}=(q_{i_1\cdots i_m})\in \Omega(\mathcal{A})$ for which $z$ is an
eigenvalue of $\mathcal {Q}$.  Then
\[ \partial
\Gamma^{R}(\mathcal {A})\subseteq \sigma
(\Omega(\mathcal{A}))\subseteq \sigma
(\hat{\Omega}(\mathcal{A}))\subseteq \Gamma^{R}(\mathcal {A}),
\]
where $\sigma (\Omega(\mathcal{A}))=\bigcup\limits_{\mathcal{D}\in
\Omega(\mathcal{A})} \sigma(\mathcal{D})$ and $\sigma
(\hat{\Omega}(\mathcal{A}))=\bigcup\limits_{\mathcal{D}\in
\hat{\Omega}(\mathcal{A})} \sigma(\mathcal{D})$.
\end{proposition}

\begin{proof} From Definition \ref{minimal ger}, \ref{equimodular} and  \ref{extended equimodular}, it is obvious that
$\Omega(\mathcal{A})\subseteq \hat{\Omega}(\mathcal{A})$ and $\sigma
(\Omega(\mathcal{A}))\subseteq \sigma
(\hat{\Omega}(\mathcal{A}))\subseteq \Gamma^{R}(\mathcal {A})$.
Hence, we next only prove $\partial \Gamma^{R}(\mathcal
{A})\subseteq \sigma (\Omega(\mathcal{A}))$.

If $z\in \mathbb{C}$ is such that $v(z)=0$, then, from the proof of
Lemma \ref{lemma 2.5}, there is a vector $y=(y_1,\ldots, y_n)^T\geq
0$, $y\neq 0$ such that $\mathcal {B}(z)y=0,$ where $\mathcal
{B}(z)$ is defined as Lemma \ref{lemma 2.5}. This implies that for
any $k\in N$
\begin{equation}\label{eq4.2} |z-a_{k\cdots k}|y_k^{m-1} =\sum\limits_{k_2,\dots,k_m\in N,\atop \delta_{kk_2\cdots
k_m}=0} |a_{kk_2\cdots k_m}|y_{k_2}\cdots y_{k_m}.
\end{equation}
Now, let $\psi_k$ satisfy \[ z-a_{k\cdots k}=|z-a_{k\cdots
k}|e^{\emph{\textbf{i}}\psi_{k}},~ k\in N,
\]
and $\mathcal {Q}=(q_{i_1\cdots i_m})\in \mathbb{C}^{[m,n]}$, where
\[q_{k\cdots k}=a_{k\cdots k} ~and ~ q_{kk_2\cdots k_m}=|a_{kk_2\cdots k_m}|e^{\textbf{\emph{i}}\psi_{k}}, ~k\in N,~\delta_{kk_2\cdots
k_m}=0.\] Hence, from Definition \ref{equimodular}, $\mathcal {Q}\in
\Omega(\mathcal{A})$. Moreover, By considering the $k$-th entry
$(\mathcal{Q}y^{m-1})_k$ of $\mathcal{Q}y^{m-1}$, we have that for
any $k\in N$,
\begin{eqnarray*}(\mathcal{Q}y^{m-1})_k&=&\sum\limits_{k_2,\ldots,k_m\in
N}q_{kk_2\cdots k_m}y_{k_2}\cdots y_{k_m}\\
&=&q_{k\cdots k}y_{k}^{m-1}+ e^{\emph{\textbf{i}}\psi_k}\left(
\sum\limits_{k_2,\dots,k_m\in N,\atop \delta_{kk_2\cdots k_m}=0}
|a_{kk_2\cdots k_m}|y_{k_2}\cdots y_{k_m}\right)\\
&=&(z-(z-a_{k\cdots k}))y_{k}^{m-1}+
e^{\textbf{\emph{i}}\psi_k}\left( \sum\limits_{k_2,\dots,k_m\in
N,\atop \delta_{kk_2\cdots k_m}=0}
|a_{kk_2\cdots k_m}|y_{k_2}\cdots y_{k_m}\right)\\
&=&zy_{k}^{m-1}+ e^{\emph{\textbf{i}}\psi_k}\left( -|z-a_{k\cdots
k}|y_{k}^{m-1}+ \sum\limits_{k_2,\dots,k_m\in N,\atop
\delta_{kk_2\cdots k_m}=0}
|a_{kk_2\cdots k_m}|y_{k_2}\cdots y_{k_m}\right)\\
&=& zy_{k}^{m-1}, ~(By~ Equality~ (\ref{eq4.2}))
\end{eqnarray*}
that is,
\[\mathcal{Q}y^{m-1}=zy^{[m-1]}.\]
Note that $y\neq 0$, then $z$ is an eigenvalue of $\mathcal{Q}\in
\Omega(\mathcal{A})$, which shows that $v(z)=0$ implies $z\in
\sigma(\Omega(\mathcal{A}))$. From Proposition \ref{boundary nec},
we have that for each point $z\in \partial \Gamma^{R}(\mathcal
{A})$, $v(z)=0$, consequently, $z\in \sigma(\Omega(\mathcal{A}))$.
Hence, $\partial \Gamma^{R}(\mathcal {A})\subseteq \sigma
(\Omega(\mathcal{A})) $. The proof is completed. \end{proof}

For the sets $\sigma (\hat{\Omega}(\mathcal{A}))$ and
$\Gamma^{R}(\mathcal {A})$, we have the following result.

\begin{proposition}\label{pro 4.3}
Let $\mathcal {A}=(a_{i_1\cdots i_m})\in \mathbb{C}^{[m,n]}$. Then
\[\sigma (\hat{\Omega}(\mathcal{A}))= \Gamma^{R}(\mathcal {A}).\]
\end{proposition}

\begin{proof} From Proposition \ref{pro 4.2}, we have $\sigma (\hat{\Omega}(\mathcal{A}))\subseteq \Gamma^{R}(\mathcal
{A})$. Hence, we only prove $\Gamma^{R}(\mathcal {A})\subseteq
\sigma (\hat{\Omega}(\mathcal{A}))$.

Let $z\in \Gamma^{R}(\mathcal {A})$. Then, from Theorem
\ref{theorem3.1}, $v(z)\geq 0$. And from the proof of Lemma
\ref{lemma 2.5}, there is a vector $y=(y_1,\ldots, y_n)^T\geq 0$,
$y\neq 0$ such that \[\mathcal {B}(z)y^{m-1}=v(z)y^{[m-1]},\] where
$\mathcal {B}(z)$ is defined as Lemma \ref{lemma 2.5}. Hence for any
$k\in N$, \begin{equation}\label{eq4.3} (|z-a_{k\cdots
k}|+v(z))y_k^{m-1} =\sum\limits_{k_2,\dots,k_m\in N,\atop
\delta_{kk_2\cdots k_m}=0} |a_{kk_2\cdots k_m}|y_{k_2}\cdots y_{k_m}
\end{equation}
Now, let $\mathcal {Q}=(q_{i_1\cdots i_m})\in \mathbb{C}^{[m,n]}$,
with
\[q_{k\cdots k}=a_{k\cdots k} ~and ~ q_{kk_2\cdots k_m}=\mu_k a_{kk_2\cdots k_m}, ~k\in N,~\delta_{kk_2\cdots
k_m}=0,\] where
\[\mu_k=\left\{\begin{array}{cc}
   \frac{\left( \sum\limits_{k_2,\dots,k_m\in N,\atop
\delta_{kk_2\cdots k_m}=0} |a_{kk_2\cdots k_m}|y_{k_2}\cdots
y_{k_m}\right)-v(z)y_k^{m-1}}{\sum\limits_{k_2,\dots,k_m\in N,\atop
\delta_{kk_2\cdots k_m}=0} |a_{kk_2\cdots k_m}|y_{k_2}\cdots
y_{k_m}},   & if ~\sum\limits_{k_2,\dots,k_m\in N,\atop
\delta_{kk_2\cdots k_m}=0} |a_{kk_2\cdots k_m}|y_{k_2}\cdots y_{k_m}>0,  \\
   1,   &if ~\sum\limits_{k_2,\dots,k_m\in N,\atop
\delta_{kk_2\cdots k_m}=0} |a_{kk_2\cdots k_m}|y_{k_2}\cdots
y_{k_m}=0.
\end{array}
\right.\] Furthermore, from Equality (\ref{eq4.3}) and the fact that
both $|z-a_{k\cdots k}|y_k^{m-1}\geq 0$ and $v(z)y_k^{m-1}\geq 0$
hold for any $k\in N$, we easily obtain $0\leq \mu_k \leq 1$ for any
$k\in N$. Hence, \[\mathcal {Q}\in \hat{\Omega}(\mathcal{A}).\] For
the tensor $\mathcal {Q}$, we have from Equality (\ref{eq4.3}) that
for any $k\in N$,
\begin{eqnarray*}|z-q_{k\cdots k}|y_k^{m-1}&=&|z-a_{k\cdots
k}|y_k^{m-1}\\
&=&\left(\sum\limits_{k_2,\dots,k_m\in N,\atop \delta_{kk_2\cdots
k_m}=0} |a_{kk_2\cdots k_m}|y_{k_2}\cdots
y_{k_m}\right)-v(z)y_k^{m-1}\\
&=&\mu_k\left(\sum\limits_{k_2,\dots,k_m\in N,\atop
\delta_{kk_2\cdots k_m}=0} |a_{kk_2\cdots k_m}|y_{k_2}\cdots
y_{k_m}\right)\\&=&\sum\limits_{k_2,\dots,k_m\in N,\atop
\delta_{kk_2\cdots k_m}=0} |q_{kk_2\cdots k_m}|y_{k_2}\cdots
y_{k_m},
\end{eqnarray*}
i.e.,
\[|z-q_{k\cdots k}|y_k^{m-1}=\sum\limits_{k_2,\dots,k_m\in N,\atop
\delta_{kk_2\cdots k_m}=0} |q_{kk_2\cdots k_m}|y_{k_2}\cdots y_{k_m}
~for ~any~k\in N.\] Now, the above expression is exactly of the form
of Equality (\ref{eq4.2}) in the proof of Proposition \ref{pro 4.2}.
Hence, similar to the proof of Proposition \ref{pro 4.2}, we have
that there is a tensor $\mathcal {P} \in \Omega(\mathcal{Q})$ such
that $z\in \sigma (\mathcal {P})$. Note that $ \mathcal {P}\in
\Omega(\mathcal{Q})$ and $\mathcal{Q}\in \hat{\Omega}(\mathcal{A})$.
Therefore, $\mathcal {P}\in \hat{\Omega}(\mathcal{A})$,
consequently, $z\in \sigma ( \hat{\Omega}(\mathcal{A}))$ and
$\Gamma^{R}(\mathcal {A})\subseteq \sigma
(\hat{\Omega}(\mathcal{A}))$. The proof is completed.
\end{proof}

From Propositions \ref{pro 4.2} and \ref{pro 4.3}, we can obtain the
following relationships.

\begin{theorem}\label{theorem 4.1}
Let $\mathcal {A}\in \mathbb{C}^{[m,n]}$. Then
\[\partial
\Gamma^{R}(\mathcal {A})\subseteq \sigma
(\Omega(\mathcal{A}))\subseteq \Gamma^{R}(\mathcal {A}).\]
\end{theorem}

\begin{remark} From Proposition \ref{pro 4.3}, we known that $\sigma
(\hat{\Omega}(\mathcal{A}))$ completely fills out $\Gamma^R(\mathcal
{A})$, that is,  $\sigma (\hat{\Omega}(\mathcal{A}))=
\Gamma^{R}(\mathcal {A})$. And from Theorem \ref{theorem 4.1}, we
known that if $\sigma (\Omega(\mathcal{A}))$ is a proper subset of
$\Gamma^{R}(\mathcal {A})$ with $\sigma (\Omega(\mathcal{A}))\neq
\partial \Gamma^{R}(\mathcal {A})$, then $\sigma
(\Omega(\mathcal{A}))$ necessarily have internal boundaries in
$\Gamma^{R}(\mathcal {A})$.
\end{remark}

\section{A numerical approximation of Minimal Ger\v{s}gorin tensor eigenvalue
inclusion set}
Unlike Ger\v{s}gorin tensor eigenvalue inclusion set
$\Gamma(\mathcal {A})$, or $\Gamma^{r^x}(\mathcal {A})$ of
(\ref{equation2.1}), Minimal Ger\v{s}gorin tensor eigenvalue
inclusion set $\Gamma^{R}(\mathcal {A})$ of a complex tensor
$\mathcal {A}$ is not easy to determine numerically generally. In
this section, for an irreducible tensor $\mathcal {A}$, we give a
numerical approximation of $\Gamma^{R}(\mathcal {A})$, which
contains $\Gamma^{R}(\mathcal {A})$. We now give a lemma, which will
be used below.

\begin{lemma} \label{lemma 4.1} Let $\mathcal {A}=(a_{i_1\cdots i_m})\in \mathbb{C}^{[m,n]}$ be
irreducible, and $v(z)$ be defined as Lemma \ref{lemma 2.5}. Then
for each $i\in N$,
\[v(a_{i\cdots i})>0.\]
Furthermore, for each $a_{i\cdots i}$ and for each real $\theta$
with $0\leq \theta < 2\pi$, let $\tilde{\gamma}_i(\theta)>0$ be the
smallest $\gamma>0$ for which
\begin{equation}\label{eq5.0}\left\{\begin{array}{l}
   v(a_{i\cdots i}+\tilde{\gamma}_i(\theta)e^{\textbf{i}\theta})=0,   \\
~ and~there ~is~ a~ sequence~ of ~\{z_j\}_{j=1}^\infty ~with\\
 \lim\limits_{j\rightarrow\infty} z_j=a_{i\cdots
i}+\tilde{\gamma}_i(\theta)e^{\textbf{i}\theta}~ such ~that~
v(z_j)<0~ for~ all~ j\geq 1.
\end{array}
\right.\end{equation}
 Then, the complex interval $[ a_{i\cdots i}+te^{\textbf{i}\theta}]$, for $0\leq t\leq \tilde{\gamma}_i(\theta) $, is
contained in $\Gamma^{R}(\mathcal {A})$ for each real $\theta\in
[0,2\pi)$, that is, the set
\[\bigcup\limits_{\theta~\in [0,2\pi) } [ a_{i\cdots i}+te^{\textbf{i}\theta}]_{t=0}^{\tilde{\gamma}_i(\theta)},~i\in N,\] is a subset of
$\Gamma^{R}(\mathcal {A})$.
\end{lemma}

\begin{proof} Let $\mathcal {B}(z)=(b_{i_1\cdots i_m})\in
\mathbb{R}^{[m,n]}$ be defined as Lemma \ref{lemma 2.5}, where,
\[b_{i\cdots i}=-|z-a_{i\cdots
i}|~and~b_{ii_2\cdots,i_m}=|a_{ii_2\cdots,i_m}|, ~ i\in N,
~\delta_{ii_2\cdots,i_m}=0.\] And let $\mathcal
{C}=(c_{i_1\cdots,i_m})\in \mathbb{R}^{[m,n]}$, where
\[c_{i\cdots i}=\mu-|z-a_{i\cdots i}|~and~c_{ii_2\cdots,i_m}=|a_{ii_2\cdots,i_m}|, ~i\in N, ~\delta_{ii_2\cdots,i_m}=0,\]
$\mu=\max\limits_{i\in N}|z-a_{i\cdots i}|$. Then $B(z)=-\mu\mathcal
{I}+\mathcal {C}$. Since $\mathcal {A}$ is irreducible, $\mathcal
{B}(z)$, also $\mathcal {C}$, is irreducible. Hence, $\mathcal {C}$
is an irreducible nonnegative tensor. From Lemma \ref{lemma 2.2},
there is a positive eigenvector of $\mathcal {C}$ corresponding to
$\rho (\mathcal {C})$. Therefore, similar to the proof of \ref{lemma
2.5}, there is a positive eigenvector of $\mathcal {B}(z)$
corresponding to $v(z)$. Moreover, by Equality (\ref{eq3.1}), that
is,
\[v(z)=\inf\limits_{x>0}\max\limits_{i\in N} \frac{(\mathcal
{B}(z)x^{m-1})_i}{x_i^{m-1}},\] we get that for any $z$, there
exists a positive vector $y$ such that for all $j\in N$,
\begin{equation} \label{eq5.1} v(z)=\frac{(\mathcal
{B}(z)y^{m-1})_j}{y_j^{m-1}}.\end{equation} Then for any $i\in N$,
take $z=a_{i\cdots i}$ and let $x=(x_1,\ldots,x_n)^T>0$ be such that
for any $j\in N$, $v(a_{i\cdots i})=\frac{(\mathcal
{B}(z)x^{m-1})_j}{x_j^{m-1}}$. In particular,
\[v(a_{i\cdots i})=\frac{(\mathcal {B}(z)x^{m-1})_i}{x_i^{m-1}}=r^{x}_i(\mathcal {A})-|a_{i\cdots i}-a_{i\cdots i}|=r^{x}_i(\mathcal {A}).\]
Since $\mathcal {A}$ is irreducible, we have $r^{x}_i(\mathcal
{A})>0$, consequently, $v(a_{i\cdots i})>0$.

Furthermore, for each $\theta$ with $0\leq \theta< 2\pi$,
$a_{i\ldots i}+te^{\emph{\textbf{i}}\theta}$, for all $t\geq 0$, is
the semi-infinite complex line. Obviously, the function $
v(a_{i\ldots i}+te^{\emph{\textbf{i}}\theta})$ is positive at $t=0$,
is continuous on this line, and is negative outside
$\Gamma^R(\mathcal {A})$ from Theorem \ref{theorem3.1}. Hence, there
is a smallest $\tilde{\gamma}_i(\theta)>0$ satisfying (\ref{eq5.0}).
And by Proposition \ref{boundary nec}, we get that $a_{i\ldots
i}+\tilde{\gamma}_i(\theta)e^{\textbf{\emph{i}}\theta} \in \partial
\Gamma^R(\mathcal {A}).$ This implies that for each real $\theta$,
[$a_{i\ldots i}$, $a_{i\ldots
i}+\tilde{\gamma}_i(\theta)e^{\textbf{\emph{i}}\theta}$] is a subset
of $\Gamma^R(\mathcal {A})$. The conclusion follows.
\end{proof}

We now give the following procedure for approximating Minimal
Ger\v{s}gorin tensor eigenvalue inclusion set $\Gamma^R(\mathcal
{A})$ for an irreducible tensor $\mathcal {A}=(a_{i_1\cdots,i_m})\in
\mathbb{C}^{[m,n]}$.

\textbf{Procedure of numerical approximation }

Step 1. determine the positive numbers $\{ v(a_{j\cdots j})\}_{j\in
N}$;

Step 2. determine the largest $\tilde{\gamma}_j(\theta)$ for $0\leq
\theta< 2\pi$, such that $a_{j\cdots
j}+\tilde{\gamma}_j(\theta)e^{\emph{\textbf{i}}\theta} \in \partial
\Gamma^R(\mathcal {A})$, i.e.,
\begin{equation}\label{eq4.4} v(a_{j\cdots
j}+\tilde{\gamma}_j(\theta)e^{\emph{\textbf{i}}\theta})=0, ~with~
v(a_{j\cdots
j}+(\tilde{\gamma}_j(\theta)+\varepsilon)e^{\emph{\textbf{i}}\theta})<
0
\end{equation}
for all sufficiently small $\varepsilon>0$;

Step 3. take the $m$ points $w_{k_{j_\theta}}=a_{j\cdots
j}+\tilde{\gamma}_j(\theta)e^{\textbf{\emph{i}}\theta} \in
\partial \Gamma^R(\mathcal {A})$ for $k_{j_\theta}=1,2,\ldots, m$, and detemine the set $ \bigcap\limits_{k_{j_\theta}=1}^m
\Gamma^{w_{k_{j_\theta}}}(\mathcal {A})$ approximating to
$\Gamma^R(\mathcal {A})$, where
\[\Gamma^{w_{k_{j_\theta}}}(\mathcal {A})=\bigcup\limits_{i\in N}\{z\in \mathbb{C}:|z-a_{i\cdots i}|\leq |w_{k_{j_\theta}}-a_{i\cdots i}|\}.\]

\begin{remark} \label{remark1} (i) To determine $v(a_{j\cdots j})$, we use Equality (\ref{eq
2.3}) with $z=a_{j\cdots j}$, i.e., \[v(a_{j\cdots
j})=-\mu+\rho(\mathcal {C}),\] where $\mu= \max\limits_{i\in
N}|a_{j\cdots j}-a_{i\cdots i}|$ and $\rho(\mathcal {C})$ is an
eigenvalue of the irreducible nonnegative tensor $\mathcal {C}$
defined as (\ref{eq 2.2}) (note that the irreducibility of $\mathcal
{C}$ is deduced by that of $\mathcal {A}$), and the following method
for calculating $\rho(\mathcal {C})$ (see \cite{Li1,Ng}), i.e., if
for an vector $x^{(0)}> 0$, let $\mathcal {D}=\mathcal {C}+h\mathcal
{I}$, where $h>0$, and let $y^{(0)} = \mathcal
{D}\left(x^{(0)}\right)^{m-1}$,
\[\begin{array}{ll}
x^{(1)}=\frac{\left(y^{(0)}\right)^{[\frac{1}{m-1}]}}{\parallel\left(y^{(0)}\right)^{[\frac{1}{m-1}]}\parallel},&
y^{(1)} = \mathcal {D}\left(x^{(1)}\right)^{m-1},\\
x^{(2)}=\frac{\left(y^{(1)}\right)^{[\frac{1}{m-1}]}}{\parallel\left(y^{(1)}\right)^{[\frac{1}{m-1}]}\parallel},&
y^{(2)} = \mathcal {D}\left(x^{(2)}\right)^{m-1},\\
\vdots &\vdots\\
x^{(k+1)}=\frac{\left(y^{(k)}\right)^{[\frac{1}{m-1}]}}{\parallel\left(y^{(k)}\right)^{[\frac{1}{m-1}]}\parallel},&
y^{(k+1)} = \mathcal {D}\left(x^{(k+1)}\right)^{m-1},~k\geq 2\\
\vdots &\vdots\\
\end{array}\]
and let
\[\underline{\lambda}_k=\min\limits_{x_i^{(k)}>0}
\frac{\left(y^{(k)}\right)_i}{\left(x_i^{(k)}\right)^{m-1}},~
\overline{\lambda}_k=\max\limits_{x_i^{(k)}>0}
\frac{\left(y^{(k)}\right)_i}{\left(x_i^{(k)}\right)^{m-1}},~k=1,2,\ldots,\]
then
\[\underline{\lambda}_1\leq \underline{\lambda}_2\leq \cdots \leq \rho(\mathcal {D})=\rho(\mathcal {C})+h
\leq  \cdots \leq \overline{\lambda}_2\leq \overline{\lambda}_1.\]
moreover, \begin{equation} \lim\limits_{k\longrightarrow
\infty}\underline{\lambda}_k= \rho(\mathcal {D})=\rho(\mathcal
{C})+h=\lim\limits_{k\longrightarrow \infty}\overline{\lambda}_k.
\end{equation}
Hence, we can obtain convergent upper and lower estimates of
$v(a_{j\cdots j})$, which do not need great accuracy for graphing
purpose, as Example \ref{example4.1} shows.

(ii) The numerical estimation of $\tilde{\gamma}_j(\theta)$. From
Lemma \ref{lemma 4.1}, there is $\tilde{\gamma}_j(\theta)$ such that
(\ref{eq4.4}) holds. Now, let $z=a_{j\cdots j}$ and
$\tilde{z}=a_{j\cdots
j}+\tilde{\gamma}_j(\theta)e^{\textbf{i}\theta}$, we have from
Proposition \ref{lemma 4.2} that
\[\tilde{\gamma}_j(\theta)\geq v(a_{j\cdots j})>0.\] Hence,
$v(a_{j\cdots j}+v(a_{j\cdots j})e^{\textbf{i}\theta})\geq 0$. If
$v(a_{j\cdots j}+v(a_{j\cdots j})e^{\textbf{i}\theta})=0$, then take
\[\tilde{\gamma}_j(\theta)=v(a_{j\cdots j})\] for which $v(a_{j\cdots j})$ can be determined
by the method of (i), otherwise, $v(a_{j\cdots j}+v(a_{j\cdots
j})e^{\textbf{i}\theta})>0$, then we increase the number
$v(a_{j\cdots j})$ to $v(a_{j\cdots j})+\Delta$, $\Delta>0$, until
$v(a_{j\cdots j}+(v(a_{j\cdots j})+\Delta)e^{\textbf{i}\theta})<0$,
and apply a bisection search to the interval $[v(a_{j\cdots j}),
~v(a_{j\cdots j})+\Delta]$ to determine $\tilde{\gamma}_j(\theta)$
satisfying (\ref{eq4.4}). Note here that estimates of
$\tilde{\gamma}_j(\theta)$ also do not need great accuracy for
graphing purpose.

(iii) It is obvious that $w_{k_{j_\theta}}$ in Step 3,
$k_{j_\theta}=1,2,\ldots, m$, are not only boundary points of
$\Gamma^R(\mathcal {A})$, but boundary points of
$\Gamma^{w_{k_{j_\theta}}}(\mathcal {A})$, and that
\[\Gamma^R(\mathcal {A})\subseteq \Gamma^{w_{k_{j_\theta}}}(\mathcal {A})\]
which shows that the larger $m$ is, the better
$\Gamma^{w_{k_{j_\theta}}}(\mathcal {A})$ approximates to
$\Gamma^R(\mathcal {A})$.
\end{remark}

\begin{example} \label{example4.1} Consider the irreducible tensor \[\mathcal {A} = [A(1,:,:),A(2,:,:),A(3,:,:)]\in
\mathbb{C}^{[3,3]},\] where
\[A(1,:,:)=\left(\begin{array}{cccc}
     2 &  0 &0  \\
      0& 0 &1    \\
      0&   0& 1  \\
\end{array}
\right),~ A(2,:,:)=\left(\begin{array}{cccc}
     0 &  0 &0  \\
      0& 2 &0    \\
      1&   0& 0  \\
\end{array}
\right), ~A(3,:,:)=\left(\begin{array}{cccc}
     1 &  1 &0  \\
      0& 1 &0    \\
      0&   0& 1  \\
\end{array}
\right).\] We next give a numerical approximation to
$\Gamma^R(\mathcal {A})$. By the part (i) of Remark \ref{remark1},
we compute $v(a_{iii})$ for $i=1,2,3$, and get
\[v(a_{111})=v(a_{222})=1.62019803,~v(a_{333})=1.43720383.\] Furthermore,
based on the entries $a_{111}=a_{222}=2$ and $a_{333}=1$, we look
for six points
\[\begin{array}{ll}
w_1=a_{111}+\tilde{\gamma}_1(0),& w_2=a_{333}-\tilde{\gamma}_3(\pi),\\
w_3=a_{111}+\textbf{i}~\tilde{\gamma}_1(\frac{\pi}{2}),&w_4=a_{111}-\textbf{i}~\tilde{\gamma}_1(\frac{3\pi}{2}),\\
w_5=a_{333}+\textbf{i}~\tilde{\gamma}_3(\frac{\pi}{2}),&
w_6=a_{333}-\textbf{i}~\tilde{\gamma}_3(\frac{3\pi}{2})
\end{array}\]
of $\partial\Gamma^R(\mathcal {A})$, which are found by the method
proposed in the part (ii) of Remark \ref{remark1}, that is,
\[\begin{array}{lll}
w_{1}=3.62019802,& w_{2}=-0.43720383,&w_3=2+\textbf{i}1.86790935,\\
w_4=2-\textbf{i}1.86790935,&
w_5=1+\textbf{i}1.81661895,&w_6=1-\textbf{i}1.81661895.
\end{array}\]
And now $\Gamma(\mathcal {A})$, $\Gamma^{w_{1}}(\mathcal {A})$,
$\Gamma^{w_{2}}(\mathcal {A})$, $\Gamma^{w_{3}}(\mathcal {A})$
$(\Gamma^{w_{3}}(\mathcal {A})=\Gamma^{w_{4}}(\mathcal {A}))$ and
$\Gamma^{w_{5}}(\mathcal {A})$ $(\Gamma^{w_{5}}(\mathcal
{A})=\Gamma^{w_{6}}(\mathcal {A}))$ are given by Figures.1, 2, 3, 4
and 5, respectively. In Figure 6, the set
$\left(\bigcap\limits_{k=1}^6 \Gamma^{w_{k}}(\mathcal {A})\right)$,
which approximates to $\Gamma^R(\mathcal {A})$, is shown with the
inner boundary, and the boundary of $\Gamma(\mathcal {A})$ is shown
with the outside. The six points $\{ w_k\}_{k=1}^6$ are plotted with
asterisks. As we can see,
$\left(\bigcap\limits_{k=1}^6\Gamma^{w_{k}}(\mathcal{A})\right)
\subset \Gamma(\mathcal {A})$, that is, the set which approximates
to Minimal Ger\v{s}gorin tensor eigenvalue inclusion set is also
contained in Ger\v{s}gorin tensor eigenvalue inclusion set. Also, it
is easy to see that more points of $\partial\Gamma^R(\mathcal {A})$
are given, the better $\Gamma^{w_{k}}(\mathcal {A})$ approximates to
$\Gamma^R(\mathcal {A})$.

\end{example}

\section{Conclusions} In this paper, we present Minimal Ger\v{s}gorin tensor eigenvalue inclusion
set $\Gamma^{R}(\mathcal {A})$, give a sufficient and necessary
condition for $\Gamma^{R}(\mathcal {A})$ by using the
$Perron-Frobenius$ theory of nonnegative tensors, and establish the
relationships between $\partial \Gamma^{R}(\mathcal {A})$, $\sigma
(\Omega(\mathcal{A}))$, $\sigma (\hat{\Omega}(\mathcal{A}))$ and
$\Gamma^{R}(\mathcal {A})$, i.e., \[\partial \Gamma^{R}(\mathcal
{A})\subseteq \sigma (\Omega(\mathcal{A}))\subseteq \sigma
(\hat{\Omega}(\mathcal{A}))= \Gamma^{R}(\mathcal {A}).\] These
results obtained are generalizations of the corresponding results of
matrices \cite {Varga} to higher order tensors. In \cite{Li-Li}, Li
et al. provided two new eigenvalue inclusion sets which are
contained in Ger\v{s}gorin eigenvalue inclusion set for tensors. An
interesting problem arises: what's the relationship between Minimal
Ger\v{s}gorin tensor eigenvalue inclusion set and the sets in
\cite{Li-Li}? In the future, we will research this problem.

\bigskip
\noindent {\bf Acknowledgments.} This work is supported by National
Natural Science Foundations of China (10961027, 71161020, 71162005)
and IRTSTYN.


\end{document}